\documentclass{amsart}
\usepackage{amsmath,amssymb}
\usepackage{graphicx,subfigure}
\usepackage{mathtools}

\usepackage{color}
\usepackage{pinlabel}

\newtheorem{theorem}{Theorem}[section]
\newtheorem{proposition}[theorem]{Proposition}
\newtheorem{corollary}[theorem]{Corollary}

\newtheorem{lemma}[theorem]{Lemma}

\theoremstyle{definition}
 \newtheorem{definition}[theorem]{Definition}

\newtheorem*{definition*}{Definition}

\theoremstyle{remark}
\newtheorem{remark}[theorem]{Remark}

\numberwithin{equation}{section}

\newcommand{\Th}{\tilde{h}}

\def\CC{\mathbb{C}}

\def\RR{\mathbb{R}}

\def\ZZ{\mathbb{Z}}

\renewcommand\SS{\mathbb{S}}
\newcommand{\cR}{{\mathcal R}}

\newcommand\minus\backslash
\newcommand{\id}{{\rm id}}

\newcommand\lan\langle
\newcommand\ran\rangle

\newcommand{\I}{{\mathrm i}}
\newcommand{\e}{{e}}
\newcommand{\dd}{{\mathrm d}}
\DeclareMathOperator\Div{div}

\renewcommand\leq\leqslant
\renewcommand\geq\geqslant
\newlength{\intwidth}

 \DeclareMathOperator\curl{curl}

\begin{document}

\title[Legendrian and electromagnetic structures]{The topology of stable electromagnetic structures and Legendrian fields on the $3$-sphere}

\author{Benjamin Bode}
\address{Instituto de Ciencias Matem\'aticas, Consejo Superior de
  Investigaciones Cient\'\i ficas, 28049 Madrid, Spain}
\email{benjamin.bode@icmat.es, dperalta@icmat.es}

\author{Daniel Peralta-Salas}

%
%
\begin{abstract}
Null solutions to Maxwell's equations in free space have the property that the topology of the electric and magnetic lines is preserved for all time. In this article we connect the study of a particularly relevant class of null solutions (related to the Hopf fibration) with the existence of pairs of volume preserving Legendrian fields with respect to the standard contact structure on the 3-sphere. Exploiting this connection, we prove that a Legendrian link can be realized as a set of closed orbits of a non-vanishing Legendrian field corresponding to the electric or magnetic part of a null solution if and only if each of its components has vanishing rotation number. Moreover, we prove that any foliation by circles (a Seifert foliation) of $\SS^3$ is isotopic to the foliation defined by a volume preserving Legendrian field with respect to the standard contact structure. We also construct a new null solution to the Maxwell's equations with the property that every positive torus knot is a closed electric line of its electric field. Finally, we prove that any (possibly knotted) toroidal surface in $\RR^3$ can be realized as a magnetic surface of a null solution to Maxwell's equations, thus implying its stability for all times. In particular, the associated volume preserving Legendrian field on $\SS^3$ exhibits a positive volume set of invariant tori of the same topological type.
\end{abstract}
\maketitle

\section{Introduction}

The time evolution of an electromagnetic field in free space is described by Maxwell's equations in $\RR^3$:
\begin{alignat}{2}
&\partial_t E=\curl B\,, \qquad &&\partial_t B=-\curl E\\
&\Div E=0\,, \qquad &&\Div B=0\,,
\end{alignat}
where we have set all constants equal to~$1$. Here, $E$ and $B$ are time-dependent vector fields that stand for the electric and magnetic part of the electromagnetic field. As usual, $\curl$ and $\Div$ denote the rotor and divergence of a vector field in $\RR^3$, respectively.

The integral curves of the electric field at fixed time $t$, i.e., the solutions to the autonomous ODE
\begin{align*}
\frac{dx(\tau)}{d\tau}=E(x(\tau),t)
\end{align*}
for some initial condition $x(0)=x_0$, are the \emph{electric lines} at time $t$. Analogously, the \emph{magnetic lines} are the integral curves of the magnetic field at each fixed time~$t$. In general, the electric and magnetic lines change their topology and bifurcate during the time evolution~\cite{ABT17}. However, for a particularly relevant class of solutions to Maxwell's equations, the so-called \emph{null solutions}, the electric and magnetic lines preserve their topology for all time.

\subsection{Stable electromagnetic structures}\label{S.null} We recall that an electromagnetic field $E(x,t),B(x,t)$ is a null solution to Maxwell's equations, if the electric and magnetic fields are orthogonal and have the same modulus, i.e., $E\cdot B=0$ and $|E|=|B|$, for all $t$. Then the time evolution of the electric and magnetic lines is given by a $1$-parameter family of diffeomorphisms $\Phi_t:\mathbb{R}^3\to\mathbb{R}^3$, where $\Phi_0$ is equal to the identity. More precisely, if the normalized Poynting field
\[
P:=\frac{E\times B}{W}
\]
is smooth everywhere, with $W:=\frac12(E^2+B^2)$ the energy density, it follows~\cite{Newcomb,kbbpsi,kpsi} that $E$ and $B$ can be written in terms of the initial electric and magnetic fields $E_0,B_0$ as
\begin{align*}
E(t)=h(t)\Phi_{t*}E_0\,,\qquad B(t)=h(t)\Phi_{t*}B_0\,,
\end{align*}
where $\Phi_{t*}$ is the push-forward of the non-autonomous flow defined by $P$, and $h$ is an explicit (in terms of $W$) non-vanishing function whose expression is not relevant for our purposes. A striking conclusion of this property is that any topological electromagnetic structure, such as knotted closed field lines or invariant tori, that $E$ or $B$ may contain at time $t=0$ is preserved for all time if the corresponding solution to Maxwell's equations is null. We shall then say that the topological structures of the initial electromagnetic field are \emph{stable} under time evolution.

In view of the previous discussion, the study of stable topological structures in electromagnetic fields is thus strongly connected to the study of null solutions to Maxwell's equations. A remarkable construction of such null fields is due to Bateman~\cite{bateman}. Roughly speaking, Bateman showed that given a map $(\alpha,\beta):\mathbb{R}^{3}\times\RR\to\mathbb{C}^2$, referred to as a pair of \textit{Bateman variables}, satisfying a certain PDE and any holomorphic function $h:\mathbb{C}^2\to\mathbb{C}$ there is a recipe to obtain a null electromagnetic field; the details of this construction will be reviewed in Section~\ref{sec:bateman}. A particular choice of the functions $(\alpha,\beta)$ and $h$ yields a solution to Maxwell's equations that is called the \emph{Hopfion}, which is intimately linked to the Hopf fibration on the 3-sphere $\SS^3$. In this solution any field line is closed and any pair of field lines forms a Hopf link, for all time. This important solution was first found by Synge~\cite{synge}, but its topological content was unveiled by Ra{\~n}ada in~\cite{R,R2} following a totally different method. It has been object of intense research during the last $25$ years, including a proposal for experimental realization~\cite{cameron}. Generalizations of the Hopfion were obtained in~\cite{kbbpsi}, where the authors constructed null solutions (following Bateman's method) with electric and magnetic lines in the shape of arbitrary torus knots. These solutions are of \textit{Hopf type}, i.e., they are constructed with the same choice of $(\alpha,\beta)$ as the Hopfion, but with different choices of the holomorphic function $h$. The interested reader may consult~\cite{ABT17} for a comprehensive account on the subject. As an aside remark, we mention that in~\cite{kpsi} the initial electromagnetic fields $E_0,B_0$ that yield null solutions under the evolution of Maxwell's equations were characterized. However it is very hard to use such a characterization effectively, so Bateman's construction remains as the only effective technique to construct null electromagnetic fields.

\subsection{A connection with contact geometry}
In spite of the clarity of Bateman's method, the construction of null electromagnetic fields encoding knots and links that are not of torus type is very challenging, and has been solved only recently~\cite{bode}. The key observation to achieve this result is a surprising connection between null solutions and contact geometry. Specifically, it was proved that any Bateman field of Hopf type is related to a pair of \emph{Legendrian vector fields} with respect to the standard contact structure on $\SS^3$ given by the kernel of the standard contact 1-form
\begin{equation}\label{eq:std}
    -y_1\dd x_1+x_1\dd y_1-y_2\dd x_2+x_2 \dd y_2,
\end{equation}
where we use $z_i=x_i+\I y_i$, $i=1,2$, as coordinates on $\mathbb{C}^2$. These two Legendrian fields are real multiples of pushforwards of the electric and the magnetic field on $\RR^3$ at time $t=0$ by the Bateman variables $(\alpha,\beta)|_{t=0}$. Throughout this paper, $\SS^3$ will be defined as the unit sphere in $\CC^2\cong \RR^4$, i.e., $\SS^3=\{(z_1,z_2)\in\CC^2:|z_1|^2+|z_2|^2=1\}$, and is endowed with the standard volume form, which is precisely the volume compatible with the standard contact form~\eqref{eq:std}.

In contrast with the usual flexibility in contact geometry, we emphasize that, in this work, isotopies of the standard contact structure are not allowed, unless they leave the kernel of the standard contact form invariant. Accordingly, by Legendrian we shall always mean an object (vector field, curve, ...) that is tangent to the kernel of the $1$-form~\eqref{eq:std}.

The goal of this article is to develop and exploit the aforementioned relation between stable electromagnetic structures (via null solutions) and Legendrian vector fields with respect to the standard contact structure. This connection is beneficial in both directions. On the one hand, it allows us to obtain electromagnetic Bateman solutions by studying Legendrian vector fields on $\SS^3$; conversely, it provides a way to construct Legendrian fields with prescribed topological structures, such as knotted periodic orbits or invariant tori, using Bateman's method. As an aside remark, we observe that Arnold in~\cite[Problem 1994-13]{Arnold} considered the problem of studying periodic orbits of Legendrian vector fields. Specifically:

\emph{Consider a particle in a magnetic field on a surface $M^2$. Study Legendrian divergence-free vector fields on $ST^*M^2$ and, in particular, their closed orbits. More generally, consider divergence-free Legendrian vector fields on $\mathbb S^3$ for some (standard?) contact structure. Does there exist a counterexample to the Seifert conjecture in this class of vector fields?}

Although we do not obtain any insight on the motion of a charged particle in a magnetic field on a surface (see Remark~\ref{R:charge}), we hope the new techniques we introduce in this work may be useful for related problems on divergence-free Legendrian fields.

\subsection{Statements of the theorems}\label{SS.state}
Our first main result is a characterization of those Legendrian vector fields on $\SS^3$ that yield null electromagnetic fields on $\RR^3$, via Bateman's construction. As we shall review in Section~\ref{sec:bateman}, the vector fields on $\RR^3$ are obtained from those on $\SS^3$ using the inverse projection map $(\alpha,\beta)|_{t=0}^{-1}$, which is a composition of the stereographic projection and a mirror reflection. The complex structure $J:T\mathbb{C}^2\to T\mathbb{C}^2$ maps any given vector field $X$ that is Legendrian with respect to the standard contact structure on $\SS^3$ to the vector field $JX$, which is the unique Legendrian vector field on $\SS^3$ (with respect to the standard contact structure) such that $|X|=|JX|$, $X\cdot JX=0$ and $(X,JX)$ is positively oriented. Here, the vector product and norm are computed using the canonical metric on $\SS^3$.

\begin{theorem}\label{P:m1}
Let $E_0$ be a real-analytic vector field on $\SS^3$ that is Legendrian with respect to the standard contact structure. Then $E_0$ and $B_0:=JE_0$ yield an electromagnetic Bateman solution on $\RR^3$ associated with the Hopf fibration if and only if $E_0$ and $B_0$ are divergence-free.
\end{theorem}

Next, using Bateman's technique, we construct non-vanishing divergence-free Legendrian vector fields with a prescribed set of periodic orbits of vanishing rotation numbers (details are provided in Section~\ref{sec:kpo}); see the aforementioned Arnold's problem for an additional motivation to study divergence-free Legendrian fields.

\begin{theorem}\label{P:m2}
Let $X$ be a non-vanishing Legendrian field on $\SS^3$ (with respect to the standard contact form). Then any periodic orbit of $X$ has vanishing rotation number. Conversely, given any Legendrian link $L=\cup_{i=1}^n L_i$ in $\SS^3$, where each $L_i$ has vanishing rotation number, there is a non-vanishing, real-analytic, divergence-free Legendrian field on $\SS^3$  (with respect to the standard contact structure) such that $L$ is isotopic to a subset of its periodic orbits.
\end{theorem}
\begin{remark}
The only condition on each $L_i$ is the vanishing rotation number, so that every Legendrian knot, no matter which value its Thurston-Bennequin invariant takes, can be realized as a periodic orbit. Note in particular, that every link type has a Legendrian representative, all of whose components have vanishing rotation number, so these assumptions do not restrict the topology of the link.
\end{remark}

As emphasized, the proof of Theorem~\ref{P:m2} makes use of Bateman's construction. Accordingly, we also obtain null solutions to Maxwell's equations with stable closed electromagnetic lines, which allows us to extend~\cite[Theorem~1]{bode} to the context of non-vanishing electromagnetic fields.

\begin{corollary}\label{cor:rot}
For every link $L=\cup_{i=1}^n L_i$ in $\RR^3$ and every (possibly empty or non-proper) subset $I\subset\{1,2,\cdots,n\}$, there is a smooth isotopy of diffeomorphisms $\Phi_t:\RR^3\to\RR^3$ and a non-vanishing electromagnetic field such that $\Phi_t(\cup_{i\in I}L_i)$ is a subset of the closed electric lines and $\Phi_t(\cup_{i\notin I}L_i)$ is a subset of the closed magnetic lines for all $t\in\mathbb{R}$.
\end{corollary}

In general, the electric and magnetic lines of the solutions obtained in Corollary~\ref{cor:rot} are not all closed. In fact, the existence of the Hopfion solution~\cite{R,R2}, where all field lines are closed, motivates the question of whether Seifert foliations of the 3-sphere other than the Hopf fibration can be realized by Bateman fields or not. In this direction, we can show that any (positive) Seifert foliation is diffeomorphic to a Legendrian foliation, with respect to the standard contact structure. In view of the contact geometry connection mentioned before, this is a first step to construct stable electromagnetic fields diffeomorphic to Seifert foliations.

\begin{proposition}\label{P:seifert}
Any Seifert foliation $S_{p,q}$, $p\geq1$, $q\geq1$, of $\SS^3$ is isotopic to the foliation defined by a non-vanishing, real-analytic, divergence-free Legendrian vector field $X$ with respect to the standard contact structure.
\end{proposition}
\begin{remark}
In view of Theorem~\ref{P:m1}, the only point that prevents us from concluding that $X$ is of Bateman-type is that we cannot check that $JX$ is divergence-free.
\end{remark}

The fields constructed in \cite{kbbpsi} have the property that they contain a set of closed field lines in the shape of a positive torus link. In that construction different torus links lead to different Bateman fields. In our next result we show that there is a single Bateman field that contains every positive torus link. To prove this claim we employ a symmetry of Bateman's construction that allows us to obtain a new pair of Bateman variables $(\tilde{\alpha},\tilde{\beta})$ from a known pair of Bateman variables $(\alpha,\beta)$ by applying complex conjugation and time reversal.

\begin{proposition}\label{P:torus}
There exists a null electromagnetic field whose electric part at time $t=0$ (and hence at any other $t\in\RR$) contains every positive torus knot as a periodic orbit.
\end{proposition}

Finally, we prove a theorem on the realization of invariant tori of arbitrary topology as magnetic (or electric) surfaces of null solutions of Bateman type in $\RR^3$ (associated with the Hopf fibration). In particular, the associated divergence-free Legendrian field on $\SS^3$ exhibits a positive volume set of invariant tori of the same topological type. The proof of this result, which is the most technically demanding of the article, consists in constructing a Bateman magnetic field with a knotted periodic orbit that is elliptic with Diophantine frequency; then a deep result by R\"ussmann~\cite{Russ} implies that the aforementioned magnetic line is accumulated by a positive volume set of invariant tori, which yields the desired magnetic surfaces. For the statement we recall that an invariant torus of a vector field is ergodic if all the orbits of the field are dense on the torus.

\begin{theorem}\label{T.torus}
Given a smooth toroidal surface $T$ in $\RR^3$, there is a null solution of Maxwell's equations with a positive volume set of magnetic surfaces isotopic to~$T$. These invariant tori of the magnetic field are ergodic and exist for all $t\in\RR$. An analogous result holds for the associated divergence-free Legendrian field (with respect to the standard contact structure) on $\SS^3$.
\end{theorem}

\subsection{Organization of the paper}
The remainder of this paper is structured as follows. Section~\ref{sec:bateman} reviews Bateman's construction and its relation to the Hopf fibration and Legendrian fields on $\SS^3$. This machinery is used in Section~\ref{sec:teo1} to prove Theorem~\ref{P:m1}. Section~\ref{sec:kpo} concerns periodic orbits of Legendrian vector fields. Theorem~\ref{P:m2} is then proved in Section~\ref{sec:teo2}, while the discussion on Legendrian Seifert foliations and the proof of Proposition~\ref{P:seifert} is presented in Section~\ref{sec:seifert}. We investigate transformations between Bateman variables via complex conjugation and time reversal and prove Proposition~\ref{P:torus} in Section~\ref{sec:torus}. The existence of null solutions exhibiting magnetic surfaces of arbitrary topology, cf. Theorem~\ref{T.torus}, is established in Section~\ref{sec:tori}. Finally, in Section~\ref{sec:int} we make several observations about integrable Bateman fields that are of independent interest and provide an explanation of the integrability properties of the fields constructed in~\cite{kbbpsi} and a connection with Rudolph's totally tangential $\mathbb{C}$-links~\cite{rudolphtt,rudolphtt2}.

\section{Bateman's construction of null electromagnetic fields}\label{sec:bateman}

In this section we describe Bateman's construction, which associates to a holomorphic function a null electromagnetic field, and its surprising connection with contact geometry.

As usual, in this work we shall represent an electromagnetic field by its Riemann-Silberstein vector
$$F=E+\I B:\mathbb{R}^{3+1}\to\mathbb{C}^3\,,$$
whose real part $E$ is the (time-dependent) electric field and whose imaginary part $B$ is the magnetic field. For an overview of the history and many uses of the Riemann-Silberstein formulation we point the reader to~\cite{RS}.

Let $\alpha, \beta:\mathbb{R}^{3+1}\to\mathbb{C}$ be two functions satisfying the PDE
\begin{equation}
\label{eq:1}
\nabla \alpha\times\nabla\beta=\I (\partial_t\alpha\nabla\beta-\partial_t\beta\nabla\alpha)\,,
\end{equation}
where $\nabla$ denotes the gradient operator with respect to the three spatial variables.
Bateman showed in~\cite{bateman} that every pair $f,g:\mathbb{C}^2\to\mathbb{C}$ of holomorphic functions defines an electromagnetic field
\begin{equation*}
F=E+\I B=\nabla f(\alpha,\beta)\times\nabla g(\alpha,\beta)\,,
\end{equation*}
which is null for all time $t$.

Using $(z_1,z_2)$ as coordinates of $\mathbb{C}^2$ and setting $h:=\partial_{z_1} f\partial_{z_2} g-\partial_{z_2} f\partial_{z_1} g$, we can rewrite $F$ as
\begin{equation}
\label{eq:bateman}
F=h(\alpha, \beta)\nabla\alpha\times\nabla\beta\,.
\end{equation}
Since $f$ and $g$ are arbitrary holomorphic functions, any holomorphic function $h$ can be used to define a Bateman field. Conversely, any (real-analytic) null field can be locally obtained via Bateman's construction~\cite{hogan} with some choice of $(\alpha,\beta)$ and holomorphic $h$. Actually, it suffices that $h$ be holomorphic on a neighbourhood of the image of $(\alpha,\beta)$.

Since these solutions to Maxwell's equations are null, the topology of their field lines is preserved for all time. That is, as recalled in Section~\ref{S.null}, the time evolution of a field line is given by transporting it along the normalized Poynting vector field.

The main difficulty in Bateman's construction is to find global solutions to Equation~\eqref{eq:1}. A particularly interesting  and fruitful choice of $(\alpha,\beta)$ is (see e.g.~\cite{kbbpsi}):
\begin{align}
\label{eq:stereo3}
\alpha(x,y,z,t)&=\frac{x^2+y^2+z^2-t^2-1+2\I z}{x^2+y^2+z^2-(t-\I)^2}\,,\nonumber\\
\beta(x,y,z,t)&=\frac{2(x-\I y)}{x^2+y^2+z^2-(t-\I)^2}\,.
\end{align}
It is easy to check that these functions solve the PDE~\eqref{eq:1}. Notice that
$$|\alpha|^2+|\beta|^2=1\,,$$ so that the image of the map $(\alpha,\beta)$ lies in the $3$-sphere. In fact, it was shown in~\cite{bode} that at $t=0$ the map $(\alpha,\beta)$ is the inverse of the stereographic projection $\SS^3\to\mathbb{R}^3\cup\{\infty\}$ followed by a mirror reflection. More precisely, denoting the composition of the stereographic projection and the mirror reflection along the plane $\{y=0\}$ by $\phi_0:\SS^3\to\mathbb{R}^3\cup\{\infty\}$, we have $\phi_0(\alpha,\beta)|_{t=0}=\id$. In fact, for any value of $t$ the map $(\alpha,\beta)$ defines a diffeomorphism from $\mathbb{R}^3\cup\{\infty\}\to\SS^3$.

It is clear that the Poynting field of a Bateman solution does not depend on the holomorphic function $h$. In the particular case of the solution $(\alpha,\beta)$ presented in Equation~\eqref{eq:stereo3}, at time $t=0$ the Poynting field is (up to a real factor) the pushforward by $\phi_0$ of the Hopf field on $\SS^3$
\begin{equation}\label{eq:hopf}
-y_1\partial_{x_1}+x_1\partial_{y_1}-y_2\partial_{x_2}+x_2\partial_{y_2}\,.
\end{equation}
Its integral curves are tangent to the fibers of the Hopf fibration, so that all field lines are closed and any pair of them forms a Hopf link (identifying $\RR^3\cup\{\infty\}$ with $\SS^3$). In what follows, a Bateman solution with the choice of $(\alpha,\beta)$ as in Equation~\eqref{eq:stereo3} will be called \emph{of Hopf type}. Note that the Hopf field~\eqref{eq:hopf} is the dual of the standard contact form on $\SS^3$, cf. Equation~\eqref{eq:std}. It then follows that at time $t=0$ both the electric and magnetic part of any Bateman field of Hopf type are pushforwards by $\phi_0$ of Legendrian fields on $\SS^3$ (with respect to the standard contact form).

The time evolution of these fields can then be understood as keeping the Legendrian fields on $\SS^3$ fixed, but varying the projection map $(\alpha,\beta)$ (or rather $(\alpha,\beta)|_{t=t_*}^{-1}:\SS^3\backslash\{(1,0)\}\to\mathbb{R}^3$) via its $t$-dependence. More precisely, notice that the fields
\begin{align}\label{eq:v1v2}
    v_1&=-x_2\partial_{x_1}+y_2\partial_{y_1}+x_1\partial_{x_2}-y_1\partial_{y_2}\,,\nonumber\\
    v_2&=-y_2\partial_{x_1}-x_2\partial_{y_1}+y_1\partial_{x_2}+x_1\partial_{y_2}\,,
\end{align}
form a basis of the standard contact plane at each point in $\SS^3$. They are isometric (via rotations) to the Hopf field in Equation~\eqref{eq:hopf}, so all their field lines are closed. It turns out that, for each fixed time $t$, they are (up to the same real factor, given explicitly in~\cite{bode}) the pushforwards of $\text{Re}(\nabla\alpha\times\nabla\beta)$ and $\text{Im}(\nabla\alpha\times\nabla\beta)$ by the map $(\alpha,\beta)|_t$. In other words, they correspond to the electric part of a Bateman solution of Hopf type with $h=1$ and $h=-\I$, respectively.

Therefore, the field
\begin{equation}
    \widetilde{F}:=h(z_1,z_2)(v_1+\I v_2)
\end{equation}
is the pushforward by $(\alpha,\beta)|_t$ of the Bateman field in Equation~\eqref{eq:bateman} with the choice~\eqref{eq:stereo3}, up to a time dependent real factor. In particular, the two fields $F$ and $\widetilde{F}$ have topologically identical field lines for all $t$, and one of them is non-vanishing if and only if the other is non-vanishing. We say that the real and imaginary part of $\widetilde{F}$ form \textit{the pair of Legendrian fields corresponding to the Bateman solution} $F$ in Equation~\eqref{eq:bateman}. Note that $\widetilde{F}$ is defined on $\SS^3$, while $F$ is defined on $\mathbb{R}^{3+1}$. We also say that a pair $(E_0,B_0)$ of Legendrian fields on $\SS^3$ is of Bateman-type if it arises as $\widetilde{F}=E_0+\I B_0$ for some Bateman solution $F$ of Hopf type.

\begin{remark}\label{rem}
As observed in \cite{bode}, the fact that the inverse of $(\alpha,\beta)|_{t=0}$ is not the stereographic projection map, but rather the composition of a stereographic projection with a mirror reflection, means that if a Bateman type Legendrian field on $\SS^3$ has a closed orbit of knot type $K$, then the corresponding electric or magnetic part in $\RR^3$ has a closed orbit, whose knot type is the mirror image of $K$. For example, the electromagnetic fields in~\cite{kbbpsi} possess field lines in the shape of positive torus knots in $\mathbb{R}^3$, but the corresponding field lines in the fields on $\SS^3$ are negative torus knots.
\end{remark}

\section{From Legendrian fields to Bateman solutions: proof of Theorem~\ref{P:m1}}\label{sec:teo1}

Let $F$ be a Bateman solution to Maxwell's equations in $\RR^{3+1}$ of Hopf type. As presented in Section~\ref{sec:bateman} (see~\cite{bode} for details), it yields a Bateman-type pair of Legendrian fields $(E_0,B_0) $ on $\SS^3$ with respect to the standard contact structure. By construction, these fields are real-analytic, divergence-free and $B_0=JE_0$ (recall the definition of $J$ in Section~\ref{SS.state}). In this section we show that, conversely, every pair of Legendrian fields with these properties is of Bateman type. As an application, we obtain a simple topological necessary condition for a vector field to be the electric or magnetic part of a Bateman solution of Hopf type.

First, we need some background on the kind of functions $\Th:\SS^3\to\mathbb{C}$ that can be extended to holomorphic functions on a neighbourhood of $\SS^3$. Let $M$ be a smooth real submanifold of $\mathbb{C}^n$. Its complex tangent space $H_pM$ at a point $p\in M$ is defined as $T_pM\cap J(T_pM)$, where $J$ is the complex structure on $\mathbb{C}^n$. $M$ is said to be a \textit{CR-manifold} if the complex dimension of $H_pM$ does not depend on $p\in M$.

The 3-sphere $\SS^3$ is an example of a CR-manifold in $\mathbb{C}^2$, where the complex tangent space is of (complex) dimension 1, given precisely by the contact distribution of the standard contact form, that is, at each point in $\SS^3$ it is generated by $v_1$ and $v_2$ introduced in Equation~\eqref{eq:v1v2}. We can think of the complex tangent space as a complex line bundle over $\SS^3$ generated by
\begin{equation}
\mathbb{L}=(-x_2+\I y_2)\frac{\partial}{\partial z_1}+(x_1-\I y_1)\frac{\partial}{\partial z_2}\,,
\end{equation}
with $\tfrac{\partial}{\partial z_i}=\tfrac{1}{2}(\tfrac{\partial}{\partial x_i}-\I\tfrac{\partial}{\partial y_i})$.
Note that $\text{Re}(\mathbb{L})=\tfrac{1}{2}v_1$ and $\text{Im}(\mathbb{L})=\tfrac{1}{2}v_2$. The \textit{tangential Cauchy-Riemann operator} is then defined as
\begin{equation}
\overline{\mathbb{L}}=(-x_2-\I y_2)\frac{\partial}{\partial\overline{z_1}}+(x_1+\I y_1)\frac{\partial}{\partial \overline{z_2}}
\end{equation}
with $\tfrac{\partial}{\partial \overline{z_i}}=\tfrac{1}{2}(\tfrac{\partial}{\partial x_i}+\I\tfrac{\partial}{\partial y_i})$.
A smooth function $f:\SS^3\to\mathbb{C}$ is called a \textit{CR-function} if it satisfies the \textit{tangential Cauchy-Riemann equations}:
\begin{equation}
\overline{\mathbb{L}}f=0\,.
\end{equation}

CR-functions play an important role in the context of holomorphic extensions of functions on submanifolds of complex space. If $\tilde{f}:\SS^3\to\mathbb{C}$ is a real-analytic function on $\SS^3$, then $\tilde{f}$ satisfies the tangential Cauchy-Riemann equations if and only if $\tilde{f}$ extends to a holomorphic function on some neighbourhood of $\SS^3$, i.e., there is a neighbourhood $U\subset \CC^2$ of $\SS^3$ and a holomorphic function $f:U\to\mathbb{C}$ such that $f|_{\SS^3}=\tilde{f}$. This result (with varying degree of generality) is attributed to Severi~\cite{severi1931, severi1958} and Tomassini~\cite{tomassini}, and holds for any real-analytic CR-manifold and its corresponding system of tangential Cauchy-Riemann equations. See~\cite{CR} for a recent account on the subject.

Now let us consider the Legendrian fields from the statement of Theorem~\ref{P:m1}. Since $E_0$ and $B_0=JE_0$ are Legendrian and real-analytic, it easily follows that they can be expressed as linear combinations of $v_1$ and $v_2$ (cf. Equation~\eqref{eq:v1v2}):
\begin{align*}
    E_0&=\text{Re}(\Th)v_1-\text{Im}(\Th)v_2\,,\nonumber\\
    B_0&=\text{Re}(\Th)v_2+\text{Im}(\Th)v_1\,,
\end{align*}
for some real-analytic function $\Th:\SS^3\to\mathbb{C}$. We compute the divergence of $E_0$ and $B_0$ and find that
\begin{align}\label{eq.div}
\text{div}(E_0)=& \text{ div}\Big(\text{Re}(\Th)v_1-\text{Im}(\Th)v_2\Big)\nonumber\\
=&-x_2\frac{\partial \text{Re}(\Th)}{\partial x_1}+y_2\frac{\partial \text{Im}(\Th)}{\partial x_1}+y_2\frac{\partial \text{Re}(\Th)}{\partial y_1}+x_2\frac{\partial \text{Im}(\Th)}{\partial y_1}\nonumber\\
&+x_1\frac{\partial \text{Re}(\Th)}{\partial x_2}-y_1 \frac{\partial \text{Im}(\Th)}{\partial x_2}-y_1\frac{\partial \text{Re}(\Th)}{\partial y_2}-x_1\frac{\partial\text{Im}(\Th)}{\partial y_2}\nonumber\\
=&\text{ Re}(\mathbb{L}(\Th))\,.
\end{align}
Similarly, we have
\begin{align}
\text{div}(B_0)=&\text{ div}\Big(\text{Re}(\Th)v_2+\text{Im}(\Th)v_1\Big)\nonumber\\
=&-x_2\frac{\partial \text{Im}(\Th)}{\partial x_1}-y_2\frac{\partial \text{Re}(\Th)}{\partial x_1}+y_2\frac{\partial \text{Im}(\Th)}{\partial y_1}-x_2\frac{\partial \text{Re}(\Th}{\partial y_1}\nonumber\\
&+x_1\frac{\partial \text{Im}(\Th)}{\partial x_2}+y_1\frac{\partial \text{Re}(\Th)}{\partial x_2}-y_1 \frac{\partial \text{Im}(\Th)}{\partial y_2} +x_1\frac{\partial \text{Re}(\Th)}{\partial y_2}\nonumber\\
=&\text{ Im}(\mathbb{L}(\Th)).
\end{align}

Therefore, $E_0$ and $B_0$ are both divergence-free if and only if the analytic  function $\Th$ satisfies the tangential Cauchy-Riemann equations. In turn, this is equivalent to the existence of a holomorphic extension of $\Th$ on a neighbourhood of $\SS^3\subset\CC^2$. Accordingly, there is a neighbourhood $U\subset \CC^2$ of $\SS^3$ and a holomorphic function $h:U\to\mathbb{C}$ such that
\begin{equation}\label{eq:rest}
h|_{\SS^3}=\Th\,.
\end{equation}



Finally, taking $(\alpha,\beta):\RR^{3+1}\to\CC^2$ as in Equation~\eqref{eq:stereo3}, we notice that Bateman's construction works if the function $h$ is holomorphic in a neighborhood of the image of $(\alpha,\beta)$, which is $\SS^3\subset\CC^2$. Accordingly, the Bateman field
\[
F=h(\alpha,\beta)(\nabla\alpha\times\nabla\beta)
\]
is a null solution to Maxwell's equations in $\RR^{3+1}$. Since $h$ satisfies Equation~\eqref{eq:rest}, $(E_0,B_0)$ is the pair of Legendrian fields on $\SS^3$ corresponding to the solution $F$. This completes the proof of the theorem.

\begin{remark}
Since the Legendrian fields of Bateman type are (up to a known real factor) pushforwards by the map $(\alpha,\beta)|_{t=0}$ of the initial electric and magnetic fields, we can interpret Theorem~\ref{P:m1} as a result comparable to~\cite[Section~3]{kpsi}, in that it provides a possibility to investigate if an electromagnetic field is null for all time by studying only its initial datum at $t=0$.
\end{remark}

\begin{remark}\label{R:charge}
There is an interesting connection between the study of divergence-free Legendrian fields and the magnetic geodesic problem~\cite{Ginz}. The magnetic geodesic problem on $\SS^2$ yields a Hamiltonian vector field on $\SS^3$ which is Legendrian with respect to the standard contact form, and divergence-free. Unfortunately, it cannot be embedded into a Bateman-type couple as in Theorem~\ref{P:m1} unless the magnetic field is identically zero on $\SS^2$. In this case one obtains a Hopf field on $\SS^3$.
\end{remark}

To conclude this section, we show that any Bateman-type non-vanishing field on $\SS^3$ is homotopic to the Hopf field $v_1$ in Equation~\eqref{eq:v1v2} via non-vanishing Legendrian fields. Since any Bateman-type field is Legendrian with respect to the standard contact structure, the claim follows from the following standard result (we include a short proof for the sake of completeness).

\begin{proposition}\label{P:m3}
Let $X$ be a non-vanishing Legendrian field on $\SS^3$ (with respect to the standard contact structure). Then $X$ is homotopic to $v_1$ via non-vanishing Legendrian fields.
\end{proposition}
\begin{proof}
As in the proof of Theorem~\ref{P:m1}, we can represent every Legendrian field by a complex-valued function $h$ on $\SS^3$, which expresses the field in the basis $\{v_1,v_2\}$. This identification provides a homeomorphism between the space of non-vanishing Legendrian fields with respect to the standard contact structure and the space of non-vanishing complex-valued functions $h:\SS^3\to\mathbb{C}\backslash\{0\}$. Since $\pi_3(\mathbb{C}\backslash\{0\})=0$, any two such maps are homotopic, which implies that the corresponding Legendrian fields are homotopic via non-vanishing Legendrian fields. In particular, $X$ is homotopic to $v_1$, which corresponds to the function $h=1$.
\end{proof}

\section{Knotted periodic orbits of non-vanishing Legendrian fields}\label{sec:kpo}

In this section we focus on non-vanishing Legendrian fields (with respect to the standard contact structure) and their periodic orbits. As explained in Section~\ref{sec:bateman}, a periodic orbit of the electric (or magnetic) part of a Bateman solution of Hopf type (at $t=0$) corresponds to a periodic orbit of the associated Legendrian field on $\SS^3$. It is thus a Legendrian knot with respect to the standard contact structure. In Theorem~\ref{P:m2} we claim that whether or not a Legendrian knot can be realized by a periodic orbit of a non-vanishing Legendrian field (and of a non-vanishing Bateman field) is only determined by one of the classical contact invariants of Legendrian knots, the rotation number~\cite[Section~3.5]{Geiges}:

\begin{definition}\label{def:rot}
Let $K$ be a null-homologous Legendrian knot in a contact manifold $(M,\xi)$ and let $\Sigma$ be an embedded orientable surface with $\partial \Sigma=K$. Given a trivialization of the restriction of the contact bundle $\xi|_{\Sigma}$ and a non-vanishing tangent vector field $V$ on $K$, we can interpret $V$ as a loop in $\mathbb{C}\backslash\{0\}$ via its coordinates with respect to the given trivialization. Then the \emph{rotation number} $\cR(K)$ is defined as the winding number of this loop, and is independent of the trivialization.
\end{definition}
\begin{remark}
We recall that any orientable $2$-plane bundle over a surface with boundary is trivial, so the trivialization in Definition~\ref{def:rot} always exists. In the case of the standard contact structure on $\SS^3$, the vector fields $\{v_1,v_2\}$ introduced in Equation~\eqref{eq:v1v2} provide a basis of such a trivialization, which is, in fact, a global basis. Accordingly, the rotation number of a Legendrian knot on $\SS^3$ (with respect to the standard contact structure) is independent of the Seifert surface $\Sigma$.
\end{remark}

In Section~\ref{sec:teo2} we prove Theorem~\ref{P:m2} and Corollary~\ref{cor:rot} using Bateman's construction and the correspondence explained in Section~\ref{sec:bateman} between Legendrian fields and null solutions of Maxwell's equations. The problem of realizing any foliation by circles on $\SS^3$ as a Bateman-type electric (or magnetic) field is explored in Section~\ref{sec:seifert}, where we prove Proposition~\ref{P:seifert}.

\subsection{Proof of Theorem~\ref{P:m2} and Corollary~\ref{cor:rot}}\label{sec:teo2}

Let $X$ be a Legendrian (with respect to the standard contact structure) and non-vanishing vector field on $\SS^3$. We argued in Section~\ref{sec:teo1} that it can be written in terms of the Hopf basis $\{v_1,v_2\}$ as
\begin{equation}
\text{Re}(h(v_1+\I v_2))
\end{equation}
for some complex-valued function $h:\SS^3\to\mathbb{C}\backslash\{0\}$. Obviously, the function $h$ is defined by expressing $X$ in terms of the aforementioned basis of the contact distribution.

Let $K$ be a periodic orbit of~$X$. Since $\SS^3$ is simply-connected, $K$ is null-homotopic and hence there is a homotopy of loops $K_s$, $s\in[0,1]$, with $K_0=K$ and $K_1$ a point $(z_1^0,z_2^0)\in\SS^3\subset\CC^2$. Then $h(K_s)$ is a homotopy in $\mathbb{C}\backslash\{0\}$ from $h(K)$ to a point $h(z_1^0,z_2^0)$. The rotation number of $K$ is, by definition, the winding number of the loop $h(K)$. Since the winding number is a homotopy invariant in $\mathbb{C}\backslash\{0\}$ and a constant loop has vanishing winding number, we conclude that $\cR(K)=0$.

Conversely, suppose that $L=\cup_{i=1}^n L_i$ is a Legendrian link on $\SS^3$ (standard contact structure) with $\cR(L_i)=0$ for all $i$. We can safely assume that $L$ is real-analytic (up to a Legendrian isotopy)~\cite{rudolphtt}, so using~\cite[Lemma~2]{bode} we can define a real-analytic function $H_i:L_i\to\mathbb{C}\backslash\{0\}$ for each component of the link, so that a nonzero tangent vector of $L_i$ at any point on $L_i$ is given by
\begin{equation}\label{eq.tang}
\text{Re}(H_i)v_1-\text{Im}(H_i)v_2\,.
\end{equation}
Since $L_i$ has vanishing rotation number, the winding number of the loop $H_i(L_i)\subset\CC\backslash\{0\}$ is zero and hence $\text{Log}(H_i)$ is a well defined real-analytic function on $L_i$.

Now, the fact that $L_i$ is a real-analytic Legendrian knot and $\text{Log}(H_i)$ is real-analytic, allows us to apply directly the extension theorem by Burns and Stout~\cite{burns}. Therefore, we conclude that $\text{Log}(H_i)$ extends to a holomorphic function $h$ on a neighbourhood $U$ of $\SS^3\subset\CC^2$, i.e., $h|_{L_i}=\text{Log}(H_i)$ for all $i$.

Finally, we take the exponential of $h$, that is
$$\e^{h}:U\to\mathbb{C}\backslash\{0\}\,,$$
which is a holomorphic function that satisfies $\e^{h}|_{L_i}=H_i$. It then follows from Equation~\eqref{eq.tang} that the real part of
\begin{equation*}
\widetilde F:=\e^{h}(v_1+\I v_2)
\end{equation*}
is a vector field $X$ on $\SS^3$ with each $L_i$ as a periodic orbit. By construction, $X$ is non-vanishing and Legendrian. Moreover, since $h$ is a holomorphic function in $U$ (and hence a CR-function), Equation~\eqref{eq.div} shows that $\Div X=0$. This completes the proof of Theorem~\ref{P:m2}.

To prove Corollary~\ref{cor:rot}, we first observe (as discussed in Section~\ref{sec:bateman}) that the image of $L$ by the map $(\alpha,\beta)|_{t=0}:\RR^3\cup\{\infty\}\to \SS^3$, which is a diffeomorphism, is a link $L'\subset\SS^3$. Since any link on $\SS^3$ is isotopic to a Legendrian link (with respect to the standard contact structure) all of whose components have vanishing rotation number, see e.g.~\cite[Section~2.7]{etnyre}, we can safely assume (up to isotopy) that $L$ is a link in $\RR^3$ whose corresponding link $L'$ satisfies the aforementioned properties.

Now we can apply the previous construction with the components $L'_i$, $i\in I$, and for the other components of $L'$ we define a real-analytic function $H_i:L'_i\to\mathbb{C}\backslash\{0\}$, $i\not\in I$, so that a nonzero tangent vector is given by
\begin{equation*}
\text{Im}(H_i)v_1+\text{Re}(H_i)v_2\,.
\end{equation*}
Arguing as before we obtain a common holomorphic extension $h:U\to\CC^2$, $h|_{L'_i}=H_i$ for all $i$, and we can then define the Bateman field on $\RR^{3+1}$
\begin{equation*}
F=\e^{h(\alpha,\beta)}(\nabla\alpha\times\nabla\beta)\,,
\end{equation*}
which is non-vanishing by construction. We know from Section~\ref{sec:bateman} that, for each $t$, the pushforward by $(\alpha,\beta)|_t$ of the electric and magnetic parts of $F$ are precisely the real and imaginary parts of $\widetilde F$, up to a non-vanishing ($t$-dependent) proportionality factor. Since $\cup_{i\in I} L'$ is a subset of periodic orbits of $\text{Re}\,\widetilde F$, and the remaining components of $L'$ are a subset of periodic orbits of $\text{Im}\,\widetilde F$, we finally conclude that, for all $t$, the electric part of $F$ has a subset of periodic orbits isotopic to $\cup_{i\in I}L_i$, and the other components correspond to periodic orbits of the magnetic part. The corollary then follows.

\subsection{Legendrian Seifert foliations}\label{sec:seifert}

The Bateman solutions in Corollary~\ref{cor:rot} exhibit a prescribed set of knotted periodic orbits, but certainly not all electric or magnetic lines are closed. In light of Ra{\~n}ada's solution~\cite{R}, which corresponds to the Hopf fibration on $\SS^3$, it is natural to ask if such topological configurations of null electromagnetic fields can occur for other link types. To be precise, we are looking for an electric and/or magnetic field with the property that all the orbits of the corresponding Bateman-type Legendrian field on $\SS^3$ are closed.

A classical result of Epstein~\cite{epstein} shows that a vector field on $\SS^3$ all of whose orbits are closed forms a Seifert foliation. Every regular fiber of a Seifert foliation of $\SS^3$ is a torus link and the two exceptional fibers always form a Hopf link. The family of Seifert foliations of $\SS^3$ is completely characterized by two integers $p$ and $q$, $q\geq1$, with $\gcd(p,q)=1$. The leaf of the foliation $S_{p,q}$ through a point $(z_1,z_2)\in \SS^3$ is given as the parametric curve
\begin{equation*}
(z_1\e^{\I p\varphi},z_2\e^{\I q\varphi})\,,
\end{equation*}
where $\varphi\in[0,2\pi)$. The leaf is a $(p,q)$-torus knot unless $z_1=0$ or $z_2=0$, in which case it is an unknot.
Any two such foliations $S_{p,q}$ and $S_{p',q'}$ with $(p,q)\neq(p',q')$, $(p,q)\neq (q',p')$ and $\gcd(p,q)=\gcd(p',q')=1$ are not equivalent, while every Seifert foliation of $\SS^3$ is equivalent to $S_{p,q}$ for some choice of $(p,q)$. We recall that two foliations are said to be equivalent if there is an orientation-preserving diffeomorphism mapping one to the other.

We infer from Theorem~\ref{P:m1} that a necessary condition for a Seifert foliation to be of Bateman-type is that all its leaves must be Legendrian with respect to the standard contact structure. See~\cite{liebermann,pang} for a general study of~\textit{Legendrian foliations}. Here we prove Proposition~\ref{P:seifert}, i.e., that every positive Seifert foliation $S_{p,q}$, $p>0$, $q>0$, is equivalent to one that is Legendrian with respect to the standard contact structure.

\begin{proof}[Proof of Proposition~\ref{P:seifert}]
A vector field $X_{p,q}$ tangent to the Seifert foliation $S_{p,q}$ of $\SS^3$ is given by
\begin{equation*}
X_{p,q}=-py_1\partial_{x_1}+px_1\partial_{y_1}-qy_2\partial_{x_2}+qx_2\partial_{y_2}\,,
\end{equation*}
where we are using coordinates $z_i=x_i+\I y_i$, $i=1,2$. Obviously, $X_{p,q}$ is divergence-free. A straightforward computation shows that $X_{p,q}$ is in the kernel of the $1$-form
\begin{align}
    \alpha_{p,q}:=&\left(-\frac{(p (x_1^2 + y_1^2) + q (x_2^2 + y_2^2)) (x_1 x_2 - y_1 y_2) x_1}{x_1^2 + y_1^2} - q \frac{(y_1 x_2 + x_1 y_2) y_1}{x_1^2 + y_1^2}\right)\dd x_1\nonumber\\
    &+\left(-\frac{(p (x_1^2 + y_1^2) + q (x_2^2 + y_2^2)) (x_1 x_2 - y_1 y_2) y_1}{x_1^2 + y_1^2} + q \frac{(y_1 x_2 + x_1 y_2) x_1}{x_1^2 + y_1^2}\right)\dd y_1\nonumber\\
    &+\left(\frac{(p (x_1^2 + y_1^2) + q (x_2^2 + y_2^2)) (x_1 x_2 - y_1 y_2) x_2}{x_2^2 + y_2^2} + p \frac{(y_1 x_2 + x_1 y_2) y_2}{x_2^2 + y_2^2}\right)\dd x_2\nonumber\\
    &+\left(\frac{(p (x_1^2 + y_1^2) + q (x_2^2 + y_2^2)) (x_1 x_2 - y_1 y_2) y_2}{x_2^2 + y_2^2} - p \frac{(y_1 x_2 + x_1 y_2) x_2}{x_2^2 + y_2^2}\right)\dd y_2\,.
\end{align}
This $1$-form is smooth everywhere (actually, real-analytic). This can be easily seen using Hopf coordinates $(z_1,z_2)=(\cos s\exp(\I\phi_1),\sin s \exp(\I\phi_2))$, $s\in[0,\pi/2]$, $\phi_{1,2}\in[0,2\pi)$, where $\alpha_{p,q}$ takes the simpler form
\[
\alpha_{p,q}=(p\cos^2 s+q\sin^2 s)\cos(\phi_1+\phi_2)ds+\sin(\phi_1+\phi_2)\sin s\cos s (qd\phi_1-pd\phi_2)\,,
\]
and $X_{p,q}$ reads as
\[
X_{p,q}=p\partial_{\phi_1}+q\partial_{\phi_2}\,.
\]

We claim that the 1-form $\alpha_{p,q}$ is contact (and hence its kernel defines a contact structure $\xi_{p,q}$) if $p>0$, $q>0$. Indeed, an easy computation in Hopf coordinates yields
\begin{align*}
\alpha_{p,q}\wedge d\alpha_{p,q}=(p+q)(p\cos^2s+q\sin^2s)\Big(-\sin s\cos s ds\wedge d\phi_1\wedge d\phi_2\Big)\\
=(p+q)(p\cos^2s+q\sin^2s)\mu_0\,,
\end{align*}
where $\mu_0$ is the standard volume-form on $\SS^3$ written in Hopf coordinates. It is then clear that $\alpha_{p,q}$ is a (real-analytic) contact form for any choice of $p>0$ and $q>0$, even if these numbers take non-integer values.

In view of the expression of $\alpha_{p,q}$ it is easy to construct an homotopy of contact forms that connects $\alpha_{p,q}$ with $\alpha_{1,1}$:
\[
\alpha(t):=\alpha_{p(t),q(t)}
\]
with $p(t)=t+p(1-t)$ and $q(t)=t+q(1-t)$, $t\in[0,1]$. Obviously $\alpha(0)=\alpha_{p,q}$ and $\alpha(1)=\alpha_{1,1}$. Moser's path method~\cite{Geiges} then implies that there is a (contact) isotopy $\Phi_t$ of diffeomorphisms that transforms the corresponding contact structures. In fact, since
\begin{equation*}
\alpha_{1,1}=-x_2\dd x_1+y_2\dd y_1+x_1 \dd x_2-y_1 \dd y_2\,,
\end{equation*}
is obtained from the standard contact form on $\SS^3$ by the rotation $(x_1,y_1,x_2,y_2)\mapsto(x_1,-x_2,y_1,y_2)$, we infer that all the contact structures $\xi_{p,q}$, $p>0$, $q>0$, are equivalent to the standard contact structure on $\SS^3$. Since the homotopy of contact forms $\alpha(t)$ is real analytic in $t$ and on $\SS^3$, the proof of Moser's theorem shows that $\Phi_t$ is a real-analytic isotopy.

Denoting by $\Phi_{p,q}$ the real-analytic diffeomorphism of $\SS^3$ such that $\xi_{p,q}$ is the pullback of the standard contact structure, we infer that the pushforward $(\Phi_{p,q})_*(X_{p,q})$ is a real-analytic Legendrian vector field with respect to the standard contact structure. Moreover, since $X_{p,q}$ is divergence-free, i.e., it preserves the standard volume form,
\[
L_{X_{p,q}}\mu_0=0\,,
\]
we conclude that $(\Phi_{p,q})_*(X_{p,q})$ preserves the volume form $(\Phi_{p,q}^{-1})^*\mu_0$. Defining the positive real-analytic function $F_{p,q}$ on $\SS^3$ as
\[
(\Phi^{-1}_{p,q})^*\mu_0=:F_{p,q}\mu_0\,,
\]
we obtain that the vector field $X:=F_{p,q}(\Phi_{p,q})_*(X_{p,q})$ is real-analytic, divergence-free and Legendrian, with respect to the standard contact structure. The proposition then follows.
\end{proof}

\begin{remark}
In principle, Moser's path method also provides us with instructions how to find the diffeomorphism $\Phi_{p,q}$. We could then check the divergence of the Legendrian orthogonal field $JX$. If $\Div(JX)=0$, then $(X,JX)$ is a pair of Legendrian fields of Bateman type and hence the topological structure is stable for all~$t$. In practice however, the differential equation that needs to be solved following Moser's approach is hard to analyze, so that at this stage the question whether Seifert foliations other than the Hopf field arise as stable structures in electromagnetic fields remains unsolved.
\end{remark}

\section{New Bateman variables and torus knots}\label{sec:torus}

In this section we construct a pair of Legendrian vector fields of Bateman type on $\SS^3$ such that the electric part realizes all positive torus knots in its set
of periodic orbits. This construction is achieved by introducing a new couple of Bateman variables, i.e., new solutions to the Equation~\eqref{eq:1}, which is presented in Section~\ref{SS.rever}.

\subsection{Antiholomorphic functions and time reversal}\label{SS.rever}
Let $\alpha:\mathbb{R}^{3+1}$ be a complex-valued function in the three spatial variables $x$, $y$ and $z$, and the temporal variable $t$. Then we denote by $\tilde{\alpha}:\mathbb{R}^{3+1}\to\mathbb{C}$ the function that is obtained from $\alpha$ by taking its complex conjugate and reversing time, i.e., $t$ is substituted by $-t$:
\[
\tilde\alpha(x,y,z,t):=\overline{\alpha(x,y,z,-t)}\,.
\]
The following lemma shows that this elementary (albeit nontrivial) transformation preserves the Bateman property.

\begin{lemma}\label{L.newBat}
Let $\alpha$ and $\beta$ be Bateman variables. Then $\tilde{\alpha}$ and $\tilde{\beta}$ as defined above are also Bateman variables.
\end{lemma}
\begin{proof}
By definition of Bateman variables, the spacetime functions $\alpha$ and $\beta$ satisfy Equation~\eqref{eq:1}. We claim that the couple $(\tilde\alpha,\tilde\beta)$ also satisfies Equation~\eqref{eq:1}. Indeed
\begin{align}
\nabla\tilde{\alpha}\times\nabla\tilde{\beta}&=\left(\overline{\nabla\alpha\times\nabla\beta}\right)|_{t\to-t}\nonumber\\
&=-\I \overline{(\partial_t\alpha\nabla\beta-\partial_t\beta\nabla\alpha)}|_{t\to-t}\nonumber\\
&=-\I (\partial_t\overline{\alpha}\nabla\overline{\beta}-\partial_t\overline{\beta}\nabla\overline{\alpha})|_{t\to-t}\nonumber\\
&=\I (\partial_t\tilde{\alpha}\nabla\tilde{\beta}-\partial_t\tilde{\beta}\nabla\tilde{\alpha})\,,
\end{align}
which completes the proof of the lemma.
\end{proof}

Next, we use this lemma to construct new null solutions to Maxwell's equations that are not of Hopf type. We recall that a function $h:U\subset\CC^2\to \CC$ is antiholomorphic if its complex conjugate $\overline{h}$ is holomorphic.

\begin{proposition}\label{P:anti}
Let $(\alpha,\beta)$ be Bateman variables such that the image of $(\overline{\alpha},\overline{\beta})$ is contained (for all $t$) in the image of $(\alpha,\beta)|_{t=0}$. Let $h:U\to\mathbb{C}$ be an antiholomorphic function on some neighbourhood $U$ of the image of $(\alpha,\beta)|_{t=0}$ in $\mathbb{C}^2$. Then $$E_0:=\text{Re}(h(\alpha,\beta)\nabla\alpha\times\nabla\beta)|_{t=0}\,, \qquad B_0:=-\text{Im}(h(\alpha,\beta)\nabla\alpha\times\nabla\beta)|_{t=0}$$
are the initial data of an electromagnetic field that is null for all time.
\end{proposition}

\begin{proof}
We first calculate
\begin{align}\label{eq.other}
E_0+\I B_0&=\overline{h(\alpha,\beta)\nabla\alpha\times\nabla\beta}|_{t=0}\nonumber\\
&=\overline{h}(\overline{\alpha},\overline{\beta})\nabla\overline{\alpha}\times\nabla\overline{\beta}|_{t=0}\nonumber\\
&=\overline{h}(\tilde{\alpha},\tilde{\beta})\nabla\tilde{\alpha}\times\nabla\tilde{\beta}|_{t=0}\,,
\end{align}
where $\tilde\alpha$ and $\tilde\beta$ are defined as before. Since $h$ is antiholomorphic on $U$, its conjugate $\overline{h}$ is holomorphic on a neighbourhood of the image of $(\overline{\alpha},\overline{\beta})$, which is the same as the image of $(\tilde{\alpha},\tilde{\beta})$. Since $\tilde{\alpha}$ and $\tilde{\beta}$ are Bateman variables, cf. Lemma~\ref{L.newBat}, the initial datum $E_0+\I B_0$ is equal to the corresponding Bateman field at $t=0$. The proposition then follows from Bateman construction and the uniqueness of solutions.
\end{proof}

\begin{remark}
It is important to emphasize that Proposition~\ref{P:anti} does not contradict Bateman's construction presented in Section~\ref{sec:bateman} (where we required holomorphicity of $h$). It is true that a pair of time dependent vector fields on $\mathbb{R}^3$ describes a Bateman field associated with the Hopf fibration if and only if the corresponding function $h$ on $\SS^3$ satisfies the tangential Cauchy-Riemann equations on $\SS^3$. However, in Proposition~\ref{P:anti} we construct null solutions to Maxwell's equations which are of Bateman type not for $(\alpha,\beta)$ from Equation~\eqref{eq:stereo3}, but for the new Bateman variables $(\tilde\alpha,\tilde\beta)$ (which are not of Hopf type). Accordingly, as shown in Equation~\eqref{eq.other}, it is $\overline{h}$ the function that must be holomorphic, and hence $h$ is antiholomorphic.
\end{remark}

We finally observe that the Bateman pair $(\alpha,\beta)$ introduced in Equation~\eqref{eq:stereo3} (Bateman solution of Hopf type) satisfies the assumption in Proposition~\ref{P:anti} because the image of $(\alpha,\beta)|_{t=t_*}$ is $\SS^3$ for all $t_*\in\mathbb{R}$. It then follows that pairs of Legendrian vector fields on $\SS^3$ whose components with respect to the Hopf basis $\{v_1,v_2\}$ are the real and imaginary part of an antiholomorphic function also correspond to Bateman fields.

\subsection{Torus knots: proof of Proposition~\ref{P:torus}}\label{sec:dense}

In this section we prove Proposition~\ref{P:torus} using the results obtained in the previous section. To this end, we consider the antiholomorphic function $h(z_1,z_2)=\overline{z_1}\overline{z_2}$. The Legendrian vector field $X$ on $\SS^3$ given by $\text{Re}(h(z_1,z_2)(v_1+\I v_2))$ can be explicitly written as
\begin{equation}
X=y_1 (x_2^2+y_2^2)\partial_{x_1}-x_1(x_2^2+y_2^2)\partial_{y_1}-y_2(x_1^2+y_1^2)\partial_{x_2}+x_2(x_1^2+y_1^2)\partial_{y_2}\,,
\end{equation}
where $z_i=x_i+\I y_i$ as usual. It is easy to check that the functions $|z_1|^2$ and $|z_2|^2$ are first integrals of the vector field $X$. Using this observation, it is elementary to compute the integral curve of $X$ through a point $(z_1,z_2)\in\SS^3$, which is given by
\begin{equation}
(z_1\e^{-\I |z_2|^2\tau},z_2\e^{\I |z_1|^2\tau})\,,
\end{equation}
for $\tau\in\RR$.

The field lines of $X$ are tangent to the tori of constant $|z_1|$ (or, equivalently, constant $|z_2|$) in $\SS^3$. The slope of the field lines is given by $-|z_2|^2/|z_1|^2$, so that the field lines are closed if and only if $|z_2|^2/|z_1|^2\in\mathbb{Q}\cup\{\infty\}$. In particular, for $|z_2|^2/|z_1|^2=p/q$ with $p,q\in\mathbb{N}$, all field lines on the corresponding torus form the torus knot $T_{p,-q}$.

Recall that for the Bateman variables $(\alpha,\beta)$ as in Equation~\eqref{eq:stereo3} we denote the inverse of $(\alpha,\beta)|_{t=0}$ by $\phi_0$. The pushforward $E_0$ of $X$ by $\phi_0$ is up to a real factor the real part of $h(\alpha,\beta)(\nabla\alpha\times\nabla\beta)|_{t=0}$. Then, by Proposition~\ref{P:anti}, $E_0$ and $B_0:=-\text{Im}(h(\alpha,\beta)(\nabla\alpha\times\nabla\beta)|_{t=0})$ form the initial datum of a Bateman field with Bateman variables $(\tilde{\alpha},\tilde{\beta})$. In particular, the topological structure of nested tori and the foliation by torus knots on a dense subset of these tori is preserved for all time. Note that $\phi_0$ is orientation reversing, since it is the composition of a stereographic projection and a mirror reflection. Therefore, the negative torus knot $T_{p,-q}$ in $\SS^3$ is transformed into its mirror image $T_{p,q}$ in $\mathbb{R}^3$. This completes the proof of the Proposition.

\section{Invariant magnetic tori: proof of Theorem~\ref{T.torus}}\label{sec:tori}

The construction of Bateman fields with prescribed periodic orbits presented in Section~\ref{sec:kpo} is based on an extension theorem by Burns and Stout~\cite{burns}, which states that a real-analytic complex valued function defined on a Legendrian link $L$ on $\SS^3$ extends to a holomorphic function on some neighbourhood of $\SS^3$. This extension is far from unique, and therefore we can obtain an entire family of Bateman fields with a desired knot $L$ as a periodic orbit. In this section, we show that among this family we can always find a vector field that exhibits a positive volume set of invariant tori that are boundaries of tubular neighbourhoods of $L$. By construction these invariant tori are magnetic surfaces of a null solution of Bateman type. Theorem~\ref{T.torus} then follows noticing that any toroidal surface $T$ in $\SS^3$ is isotopic to the boundary of some knot $L$, which in turn is isotopic to an analytic Legendrian knot.

To this end, we first recall that by Rudolph's theorem~\cite{rudolphtt, rudolphtt2, rudolph} (cf. Section~\ref{SS.tot}) any real-analytic Legendrian link $L$ on $\SS^3$ arises as the totally tangential intersection $F^{-1}(0)\cap\SS^3$ of the vanishing set of a holomorphic function $F:\CC^2\to\mathbb{C}$. This function $F$ is not unique. The following technical lemma shows that the function $F$ can be chosen so that its gradient on $L$ is prescribed. In the proof we use the vector fields
\begin{align*}
&v_3=x_1\partial_{x_1}+y_1\partial_{y_1}+x_2\partial_{x_2}+y_2\partial_{y_2}\,,\\
&v_4=-y_1\partial_{x_1}+x_1\partial_{y_1}-y_2\partial_{x_2}+x_2\partial_{y_2}\,,
\end{align*}
which together with the Legendrian fields $v_1$ and $v_2$ form an orthonormal basis of $T_p\RR^4$ for each point $p\in\SS^3$.


\begin{lemma}\label{lem:rudolph}
Let $L$ be a real-analytic Legendrian link in $\SS^3$. Then there is a holomorphic function $F:U\to\mathbb{C}$ on some neighbourhood $U$ of $\SS^3$ such that $L$ is a component of $F^{-1}(0)\cap\SS^3$, all points on $L$ are tangential intersections, and
$$\nabla\text{Re}(F)(x_1,y_1,x_2,y_2)=(x_1,y_1,x_2,y_2)$$
for all $(x_1,y_1,x_2,y_2)\in L$, where $\nabla$ is the Euclidean gradient with respect to $x_1$,$y_1$,$x_2$ and $y_2$.
\end{lemma}
\begin{remark}
Note that the equality $\nabla\text{Re}(F)(x_1,y_1,x_2,y_2)=(x_1,y_1,x_2,y_2)$ for all $(x_1,y_1,x_2,y_2)\in L$ is equivalent to $\tfrac{\partial F}{\partial z_i}|_L=\overline{z_i}$, $i=1,2$.
\end{remark}
\begin{proof}
By Rudolph, since $L$ is real-analytic and Legendrian, there is a holomorphic function $\widetilde{F}:\mathbb{C}^2\to\mathbb{C}$ such that $L$ is the tangential intersection of $\widetilde{F}^{-1}(0)$ and~$\SS^3$. Accordingly, both $\nabla\text{Re}(\widetilde{F})$ and $\nabla\text{Im}(\widetilde{F})$ are in $\text{span}(v_3,v_4)$, i.e., $\nabla\text{Re}(\widetilde{F})=\lambda_1 v_3+\lambda_2 v_4$ and $\nabla\text{Im}(\widetilde{F})=\lambda_1 v_4-\lambda_2 v_3$ (because $\widetilde{F}$ is holomorphic), for some real-analytic couple of functions $(\lambda_1,\lambda_2):L\to \mathbb{R}^2\backslash\{(0,0)\}$.

Now, for every point $p\in L$ we define $g:L\to\mathbb{C}\backslash\{0\}$ as
\begin{equation}
g(p):=\frac{\lambda_1+\I\lambda_2}{\lambda_1^2+\lambda_2^2}\,,
\end{equation}
which is a real-analytic complex valued function on $L$, and by Burns and Stout's theorem, it has a holomorphic extension on some neighbourhood $U$ of $\SS^3$, which we denote by $\widetilde{g}$.

Finally, setting $F:=\widetilde{g}\widetilde{F}$ we easily infer that
\begin{align}
    \nabla\text{Re}(F)(x_1,y_1,x_2,y_2)&=(\text{Re}(\widetilde{g})\lambda_1+\text{Im}(\widetilde{g})\lambda_2)v_3+(\text{Re}(\widetilde{g})\lambda_2-\text{Im}(\widetilde{g})\lambda_1)v_4\nonumber\\
    &=v_3=(x_1,y_1,x_2,y_2)
\end{align}
for all $(x_1,y_1,x_2,y_2)\in L$. Since $L$ is a component of the set of tangential intersections of $F^{-1}(0)$ and $\SS^3$, the lemma follows.
\end{proof}

Let $L$ be a real-analytic Legendrian knot in $\SS^3$. Our goal is to construct a Legendrian field $B_0$ on $\SS^3$ corresponding to the magnetic part of a Bateman type solution (at $t=0$) such that $L$ is an elliptic periodic orbit of $B_0$ with Diophantine frequency~\cite[Section~2]{Chicone}. This means that the monodromy matrix defined by the flow of $B_0$ on $L$ has two purely imaginary eigenvalues $\exp(\pm \I\omega)$ and $\frac{\omega}{2\pi}$ is a Diophantine number. We recall that an irrational number $w$ is Diophantine if there exist constants $\gamma>0$ and $\tau>2$ such that $\Big|w-\frac{p}{q}\Big|\geq \frac{\gamma}{|q|^\tau}$ for any $p\in\ZZ$ and $q\in \ZZ_0$. As we shall see later, this is related to the existence of invariant tori of $B_0$ that accumulate over $L$.

Following Bateman construction used in Section~\ref{sec:kpo}, let $U$ be a neighbourhood of $\SS^3\subset\mathbb{C}^2$ and let $h:U\to \mathbb{C}$ be a holomorphic function such that $L$ is a periodic orbit of the magnetic part $\tilde B_0$ of the corresponding Bateman field, i.e.,
$$\tilde B_0:=\text{Im}(h(v_1+\I v_2))\,,$$
when restricted to $L$, is a unit tangent vector $T$. Let us take a holomorphic function $F:U\to\mathbb{C}$ as in Lemma~\ref{lem:rudolph}, whose zeros intersect $\SS^3\subset\mathbb{C}$ tangentially in $L$ and $\tfrac{\partial F}{\partial z_i}|_L=\overline{z_i}$, $i=1,2$, and an arbitrary holomorphic function $g:U\to\mathbb{C}$. It is clear that the Bateman field corresponding to $h+gF$, which we denote by $B_0$,
\[
B_0:=\text{Im}((h+gF)(v_1+\I v_2))\,,
\]
also has $L$ as a periodic orbit, and the gradient of $h+gF$ on $L$ satisfies
\begin{align}
\nabla \text{Re}(h+gF)&=\nabla\text{Re}(h)+\text{Re}(g)(x_1,y_1,x_2,y_2)_{\circ,p}-\text{Im}(g)(-y_1,x_1,-y_2,x_2)_{\circ,p}\,,\\
\nabla \text{Im}(h+gF)&=\nabla\text{Im}(h)+\text{Im}(g)(x_1,y_1,x_2,y_2)_{\circ,p}+\text{Re}(g)(-y_1,x_1,-y_2,x_2)_{\circ,p}\,.
\end{align}
Here and in what follows we use the notation that for any tangent vector $V=\xi_1\partial_{x_1}+\xi_2\partial_{y_1}+\xi_3\partial_{x_2}+\xi_4\partial_{y_2}$ in $T_p\mathbb{R}^4$ we write $V^{\circ}$ for the point $(\xi_1,\xi_2,\xi_3,\xi_4)\in\mathbb{R}^4$, and to identify a point in $\mathbb{R}^4$ with a tangent vector at a given point $p\in\mathbb{R}^4$ we write $(\xi_1,\xi_2,\xi_3,\xi_4)_{\circ,p}=\xi_1\partial_{x_1}+\xi_2\partial_{y_1}+\xi_3\partial_{x_2}+\xi_4\partial_{y_2}$.

Let $N(L)$ denote a tubular neighbourhood of $L$ in $\SS^3$ and $NL$ its normal bundle in $\SS^3$. Let $N_{\mathbb{R}^4}(L)$ and $N_{\mathbb{R}^4}L$ denote the corresponding objects when $L$ is considered as a submanifold of $\mathbb{R}^4$. Since $L$ is Legendrian, $N_pL$, $p\in L$, has a basis that is given by $v_4$ (the dual to the standard contact form) and a Legendrian orthogonal $Z$ to $L$, given explicitly by $\text{Re}(h(v_1+\I v_2))$ (or, equivalently, $\text{Re}((h+gF)(v_1+\I v_2))$). Likewise, $\{v_4,Z,v_3\}$ forms an orthonormal basis of $N_{\mathbb{R}^4,p}L$, the normal bundle of $L$ in $\mathbb{R}^4$ at $p\in L$.

To obtain the monodromy matrix of the magnetic field $B_0$ on its periodic orbit $L$ we need to compute the Jacobian matrix of $B_0$ on $L$ (actually only its normal component on $NL$ is needed), and to do this we introduce suitable coordinate systems on $N(L)$ and $N_{\mathbb{R}^4}(L)$. First we use the exponential map in $\mathbb{R}^4$ with the Euclidean metric; ultimately, this is not the coordinate system that we are interested in, but it will simplify our calculations quite a bit.

Let $\alpha\in \SS^1_L$ be an arc-length parametrization of the knot $L$ and $\gamma(\alpha)$ an analytic embedding of $\SS^1_L$ into $\SS^3$ whose image is $L$, i.e., $\tfrac{\partial\gamma}{\partial\alpha}=T$. There is no loss of generality in assuming that $|L|=1$. Obviously $\{T,v_4,Z,v_3\}$ forms an orthonormal basis of $T_p\mathbb{R}^4$ for every $p\in L$. Since the geodesics in $\mathbb{R}^4$ are straight lines, the tangent space at a point $p$ in $N_{\mathbb{R}^4}(L)$ has an orthonormal basis given by $$\mathcal{B}=\{V_1,V_2,V_3,V_4\}:=\{T(\alpha),v_4(\alpha),Z(\alpha),v_3(\alpha)\}\,,$$
where $\gamma(\alpha)$ is the closest point to $p$ on $L$. With a slight abuse of notation, in this expression $T(\alpha)$ (at some point $p$ away from the knot $L$) is the parallel transport of $T(\alpha)$ (on the knot) along the unique geodesic in $\mathbb{R}^4$ from $\gamma(\alpha)$ to $p$, and analogously for $v_4(\alpha)$, $Z(\alpha)$ and $v_3(\alpha)$. We emphasize that $v_3(\alpha)$ is normal to $\SS^3$ at $\gamma(\alpha)$, but not normal at other points in $N_{\mathbb{R}^4}(L)$, and $\{T(\alpha),v_4(\alpha),Z(\alpha)\}$ are tangent to $\SS^3$ only at $\gamma(\alpha)$.

Noticing that $V_1\cdot \nabla\big(|h+gF|^2\big)=0$ at each point of $L$ because $|h+gF|^2|_{L}=1$, a straightforward computation (e.g. with the computer software Mathematica) using the Cauchy-Riemann equations and Lemma~\ref{lem:rudolph} yields the following identities at any point on $L$ (parametrized by the $\alpha$-coordinate):
\begin{align}\label{eq:old}
(\partial_{V_2}B_0)\cdot V_2=&0\\
(\partial_{V_2}B_0)\cdot V_3=&1+\text{Re}(g)\text{Re}(h)+\text{Im}(g)\text{Im}(h)\nonumber\\
&+\text{Re}(h)(\nabla\text{Re}(h)\cdot v_3)-\text{Im}(h)(\nabla\text{Re}(h)\cdot v_4)\nonumber\\
=:&G(\alpha)\\
(\partial_{V_3}B_0)\cdot V_2=&-1\\
(\partial_{V_3}B_0)\cdot V_3=&0.\label{eq:old2}
\end{align}
Since $g$ is an arbitrary function and $h|_L\neq 0$, we can choose $G(\alpha)$ to be any real-analytic function on $L$. Here and in what follows, given two vector fields $W_1$ and $W_2$ on $\RR^4$, $\partial_{W_1} W_2$ denotes the vector field $(W_1\cdot \nabla) W_2$, where the action of $W_1\cdot \nabla$ is understood componentwise ($\nabla$ and $\cdot$ are computed using the Euclidean metric in~$\RR^4$).

Next we observe that at each point on $L$ the vectors $-Z$ and $v_4$ form an orthonormal basis of $NL$ at that point. Then this defines coordinates on the normal bundle of $L$ in $\SS^3$ via the identification with $\SS^1\times\mathbb{R}^2$ that maps
\begin{align}\label{eq:basis}
(\alpha,x,y)\in\SS^1\times\mathbb{R}^2&\mapsto(\alpha,-xZ(\alpha),yv_4(\alpha))\,.
\end{align}
Since the exponential map $\exp_{\SS^3}:N_{\SS^3}L\to N_{\SS^3}(L)$ is a diffeomorphism on a neighbourhood of the zero section of the normal bundle, the chosen coordinates on the normal bundle define a set of coordinates on a tubular neighbourhood of $L$ in $\SS^3$. As the geodesics on $\SS^3$ are great circles, the exponential map is explicitly given by
\begin{align}
\exp_{\SS^3}(\alpha,x,y)=&\cos\left(\sqrt{x^2+y^2}\right)\gamma(\alpha)\nonumber\\
&+\frac{\sin\left(\sqrt{x^2+y^2}\right)}{\sqrt{x^2+y^2}}\left(-xZ(\alpha)+yv_4(\alpha)\right)^{\circ}.
\end{align}

In what follows we aim to calculate the pushforwards of $\partial_{\alpha}$, $-Z(\alpha)$ and $v_4(\alpha)$ (the holonomic basis on $N_{\SS^3}L$ associated to the coordinates $(\alpha,x,y)$) by the exponential map, express $B_0$ in terms of these coordinates and compute the derivatives of its components. In order to simplify our calculations and keep working in 4-dimensional Euclidean coordinates, we extend $\exp_{\SS^3}$ to a map $\widetilde{\exp}$ on $N_{\mathbb{R}^4}L$, not just on $NL$. We define
\begin{equation}
\widetilde{\exp}(\alpha,x,y,z):=(z+1)\exp_{\SS^3}(\alpha,x,y)\,.
\end{equation}
Note that for $z=0$, we have the usual exponential map $\exp_{\SS^3}$. Using Mathematica we can calculate the pushforward of $\partial_\alpha$, $-Z(\alpha)$, $v_4(\alpha)$ and $v_3(\alpha)$ by $\widetilde{\exp}$ and obtain a new ordered holonomic basis $\mathcal{B'}=\{w_1,w_2,w_3,w_4\}$, corresponding to new coordinates $(\bar\alpha,\bar x,\bar y,\bar z)$ of $N_{\mathbb{R}^4}(L)$. By construction we have
\[
w_1=\partial_{\bar\alpha}\,,\qquad w_2=\partial_{\bar x}\,,\qquad w_3=\partial_{\bar y}\,,\qquad w_4=\partial_{\bar z}\,,
\]
and $\bar\alpha=\alpha$ on $L$, which in the new coordinates is defined by $\bar x=\bar y=\bar z=0$.

By expressing each $w_i$ as a linear combination of the $V_j$ we define the matrix $M$ that satisfies $\mathcal{B}'=M\mathcal{B}$. For this it is useful to recall that the basis vectors in $\mathcal{B}$ were defined in such a way that on the knot $L$ they form conjugate pairs with respect to the standard complex structure:
\begin{align*}
Z(\alpha)&=-JT(\alpha)\,,\\
v_4(\alpha)&=Jv_3(\alpha)\,,\\
\gamma(\alpha)_{\circ,\gamma(\alpha)}&=v_3(\alpha)\,,
\end{align*}
as well as
\begin{align*}
\frac{\partial v_4^{\circ}(\alpha)}{\partial \alpha}&=\frac{\partial (J(\gamma)_{\circ,\gamma(\alpha)})^{\circ}}{\partial \alpha}&\nonumber\\
&=J\frac{\partial \gamma(\alpha)}{\partial \alpha}=JT(\alpha)=-Z(\alpha)\,.
\end{align*}

We know that $\mathcal{B}$ forms an orthonormal basis of the tangent space of $\mathbb{R}^4$ on $L$. It then follows that
\begin{align}
\frac{\partial Z(\alpha)}{\partial \alpha}\cdot v_4(\alpha)&=\frac{\partial (v_4(\alpha)\cdot Z(\alpha))}{\partial \alpha}-Z(\alpha)\cdot\frac{\partial v_4(\alpha)}{\partial \alpha}\nonumber\\
&=Z(\alpha)\cdot Z(\alpha)=1\,,\\
\frac{\partial Z(\alpha)}{\partial \alpha}\cdot Z(\alpha)&=\frac{1}{2}\frac{\partial (Z(\alpha)\cdot Z(\alpha))}{\partial \alpha}=0\,.
\end{align}
Hence
\begin{equation}
\frac{\partial Z}{\partial \alpha}=v_4(\alpha)+F_1T(\alpha)+F_2v_3(\alpha)
\end{equation}
for some functions $F_1$, $F_2:\SS^1_L\to\mathbb{R}$.

Using these identities we can compute the matrix $M$ (which is invertible) and find that its entries $m_{ij}$, $i,j=1,2,3,4$, satisfy the following properties on the knot, i.e., when $x=y=z=0$:
\begin{align*}
m_{22}&=0\,,\\
m_{23}&=-1\,,\\
m_{32}&=1\,,\\
m_{33}&=0\,,\\
m_{11}&=m_{44}=1\,,
\end{align*}
and all other entries are $0$.

Then we can express the magnetic field in coordinates $(\bar\alpha,\bar x,\bar y,\bar z)$ as $B_0=\sum_{i=1}^4B_{0,i}w_i$ (recall that $\mathcal B'$ forms a holonomic basis in these coordinates), with
$$B_{0,i}=\sum_{j=1}^4\left(M^{-1}\right)_{ji}B_0\cdot V_j\,.$$

In order to compute the first order Taylor expansion of the vector field $B_0$ expressed in the coordinates $(\bar\alpha,\bar x,\bar y,\bar z)$, we need the derivatives $\tfrac{\partial B_{0,i}}{\partial w_j}$ evaluated at $\bar x=\bar y=\bar z=0$ (i.e., on $L$). Note that $w_4$ is not relevant to us, since we know that $B_0$ is tangent to $\SS^3$. We write $c_{ij}$ for the entries of the matrix $\left(M^{-1}\right)_{ij}$.

In general, we have
\begin{equation}\label{eq:B}
\partial_{w_j}B_{0,i}=\sum_{k=1}^4\Big((\partial_{w_j}c_{ki})(B_0\cdot V_k)+c_{ki}(\partial_{w_j}B_0)\cdot V_k+c_{ki}B_0\cdot(\partial_{w_j}V_k)\Big)\,.
\end{equation}
We are only interested in $i,j=2,3$, which corresponds to the normal directions to $L$ in $\SS^3$.
It helps that on $L$ we have
\begin{align*}
B_0\cdot V_1&=B_0\cdot T=1\,,\\
B_0\cdot V_2&=B_0\cdot V_3=B_0\cdot V_4=0\,,
\end{align*}
because $L$ is Legendrian. Therefore, the only relevant derivatives of $c_{ki}$ in Equation~\eqref{eq:B} are (at $\bar x=\bar y=\bar z=0$):
\begin{align*}
\partial_{w_2}c_{12}&=0\,,\\
\partial_{w_3}c_{12}&=-1\,,\\
\partial_{w_2}c_{13}&=1\,,\\
\partial_{w_3}c_{13}&=0\,.
\end{align*}

Furthermore, for $\bar x=\bar y=\bar z=0$
\begin{align*}
c_{11}&=c_{44}=1\,,\\
c_{23}&=1\,,\\
c_{32}&=-1\,,\\
c_{22}&=c_{33}=0
\end{align*}
and all other entries are $0$.

Additionally,
\begin{equation}
\partial_{V_i}V_j=0\qquad \text{ if }i,j\neq 1,
\end{equation}
since $V_j$ is defined by parallel transport along straight lines in $\mathbb{R}^4$ that are normal to $L$.

Consider now the last term of the sum in Equation~\eqref{eq:B}, that is
\begin{align*}
\sum_{k=1}^4c_{ki}B_0\cdot(\partial_{w_j}V_k)&=\sum_{k=1}^4c_{ki}B_0\cdot(\sum_{\ell=1}^4m_{j\ell}\partial_{V_\ell}V_k)\nonumber\\
&=c_{1i}B_0\cdot(m_{j1}\partial_{V_1}V_1)\,,
\end{align*}
which is zero for both $j=2$ or $j=3$ because $m_{21}=m_{31}=0$.

Now note that by definition of the matrix $M$, $\partial_{w_j}B_0=\sum_{i=1}^4m_{ji}\partial_{V_i}B_0$. Using then Equations~\eqref{eq:old}--\eqref{eq:old2}, the values of the derivatives $\partial_{w_j}c_{ki}$ and the entries $c_{ij}$ on $L$, we infer that when $\bar x=\bar y=\bar z=0$ the following identities hold:
\begin{align}
a_{22}:=\partial_{w_2}B_{0,2}=&\sum_{k=1}^4c_{k2}m_{23}(\partial_{V_3}B_0)\cdot V_k\nonumber\\
=&c_{32}m_{23}(\partial_{V_3}B_0)\cdot V_3\nonumber\\
=&0\,,
\end{align}
\begin{align}
a_{23}:=\partial_{w_3}B_{0,2}=&(\partial_{w_3}c_{12})(B_0\cdot V_1)+\sum_{i,k=1}^4c_{k2}m_{3i}(\partial_{V_i}B_0)\cdot V_k\nonumber\\
=&-1+c_{32}m_{32}(\partial_{V_2}B_0)\cdot V_3\nonumber\\
=&-(1+G(\bar\alpha))\,,
\end{align}
\begin{align}
a_{32}:=\partial_{w_2}B_{0,3}=&(\partial_{w_2}c_{13})(B_0\cdot V_1)+\sum_{i,k=1}^4c_{k3}m_{2i}(\partial_{V_i}B_0)\cdot V_k\nonumber\\
=&1+c_{23}m_{23}(\partial_{V_3}B_0)\cdot V_2\nonumber\\
=&2\,,
\end{align}
and
\begin{align}
a_{33}:=\partial_{w_3}B_{0,3}=&\sum_{i,k=1}^4c_{k3}m_{3i}(\partial_{V_i}B_0)\cdot V_k\nonumber\\
=&c_{23}m_{32}(\partial_{V_2}B_0)\cdot V_2\nonumber\\
=&0\,.
\end{align}

Therefore, recalling that the basis $\mathcal B'$ is holonomic, the vector field $B_0$ in the local coordinates $(\bar\alpha,\bar x,\bar y)$ defined in a neighborhood of~$L$ in $\SS^3$ takes the form
\begin{align*}
B_{0,1}(\bar\alpha,\bar x,\bar y)&=1 +h.o.t.\,,\\
B_{0,2}(\bar\alpha,\bar x,\bar y)&=a_{22}\bar x+a_{23}\bar y+h.o.t.\,,\\
B_{0,3}(\bar\alpha,\bar x,\bar y)&=a_{32}\bar x+a_{33}\bar y+h.o.t\,,
\end{align*}
where h.o.t. refers to higher order terms in $\bar x$ and $\bar y$.

It is clear that $L=\{\bar x=\bar y=0\}$ is a periodic orbit of $B_0$, in particular $B_0|_{\bar x=\bar y=0}=\partial_{\bar \alpha}$. To analyze its stability the usual tool is to study the normal variational equation (NVE) on $L$, which is the ODE
\[
\dot\xi =DB_0|_{(t,0,0)}\xi\,,
\]
where $\xi=(\xi_1,\xi_2)$ (cf.~\cite[Section]{Chicone}). Setting the $2\times 2$ matrix
\[
A:=\left(
     \begin{array}{cc}
       a_{22} & a_{23} \\
       a_{32} & a_{33} \\
     \end{array}
   \right)=
   \left(
     \begin{array}{cc}
       0 & -(1+G(t)) \\
       2 & 0 \\
     \end{array}
   \right)\,,
\]
the NVE is given by $\dot\xi=A\xi$. A convenient choice of the function $G$ which allows us to integrate explicitly this ODE is
\[
G(\alpha)=\frac{\omega^2}{2}-1
\]
for some constant $\omega$ that will be fixed later. This yields that the eigenvalues of the matrix $A$ are $\pm\I\omega$. The monodromy matrix $\mathcal M$ is then given~\cite[Section~2]{Chicone} by the fundamental solution $\Phi(t)$ of the NVE evaluated at $t=1$, which is the period of the periodic orbit (recall that we assumed that $|L|=1$). After a straightforward computation we then infer that
\[
\mathcal M=\left(
             \begin{array}{cc}
               \cos(\omega) & -\frac{\omega}{2}\sin(\omega) \\
               \frac{2}{\omega}\sin(\omega) & \cos(\omega) \\
             \end{array}
           \right)\,,
\]
whose eigenvalues are $\exp(\pm\I\omega)$. By definition, these are the Floquet multipliers of the vector field $B_0$ on its periodic orbit $L$, so if we take $\omega$ so that $\frac{\omega}{2\pi}$ is Diophantine, we conclude that $L$ is an elliptic periodic orbit of $B_0$ with Diophantine frequency. In turn, since $B_0$ is analytic and divergence-free, a straightforward application of R\"ussmann's theorem implies that the periodic orbit $L$ is stable, which means that there is a positive volume set of invariant tori of $B_0$ that accumulate on $L$; these tori are the boundaries of tubular neighborhoods of the knot. This completes the proof of Theorem~\ref{T.torus}.

\begin{remark}
R\"ussmann's theorem is stated in terms of local diffeomorphisms of $\RR^2$, so to apply it we need to take an analytic transverse section $\Sigma\cong\RR^2$ at some point on $L$ and the Poincar\'e return map $\Pi:D\subset\Sigma\to \Sigma$ defined by the flow of $B_0$, with $D$ a disk of small radius centered at $0$. Since $L$ is a periodic orbit, the origin is a fixed point of $\Pi$; moreover it is standard that this diffeomorphism is area-preserving and analytic (because $B_0$ is divergence-free and analytic). The fact that $L$ is an elliptic periodic orbit of $B_0$ with Diophantine frequency is transferred to the fixed point $0$ of $\Pi$, and hence we can apply Russmann's theorem.
\end{remark}



\section{Final remarks on integrable Bateman fields}\label{sec:int}

In this final section we explore the integrability properties (i.e., the existence of first integrals) of Bateman solutions of Hopf type. Let us first recall that the Bateman fields in~\cite{kbbpsi}, which contain torus knots, correspond to the holomorphic function $h=pqz_1^{p-1}z_2^{q-1}$, and can be equivalently written as
\begin{equation}
    F=\nabla f(\alpha,\beta)\times\nabla g(\alpha,\beta)\,,
\end{equation}
with $(\alpha,\beta)$ as in Equation~\eqref{eq:stereo3} and $f=z_1^p$ and $g=z_2^q$.

Besides the fact that these fields have periodic orbits in the shape of a $(p,q)$-torus knot, these fields possess a surprising integrability property, that is, $E=\text{Re}(F)$ is tangent to the level sets of $\text{Im}(fg)$ and $B=\text{Im}(F)$ is tangent to the level sets of $\text{Re}(fg)$ (both functions are first integrals of the corresponding fields). This can be proved by a direct calculation. In what follows we want to explain this integrability property in a more general way that relates the electromagnetic fields to the concept of totally tangential $\mathbb{C}$-links.

\subsection{Legendrian fields with prescribed first integrals}

First, we introduce a method (of interest in itself) to construct Legendrian vector fields on $\SS^3$ (with respect to the standard contact structure) with a prescribed first integral.

\begin{lemma}\label{L.fi}
Let $G_{\mathbb{R}}:\SS^3\to\mathbb{R}$ be a smooth function and define the complex-valued function
\begin{equation}\label{eq:htang}
    h(z_1,z_2)=\left(\frac{\partial G_{\mathbb{R}}}{\partial x_2}-\I \frac{\partial G_{\mathbb{R}}}{\partial y_2}\right)(x_1-\I y_1)-\left(\frac{\partial G_{\mathbb{R}}}{\partial x_1}-\I \frac{\partial G_{\mathbb{R}}}{\partial y_1}\right)(x_2-\I y_2)\,.
\end{equation}
Then $G_{\mathbb{R}}$ is a first integral of the Legendrian field $\text{Im}(h(v_1+\I v_2))$.
Furthermore, if $G_{\mathbb{R}}$ extends to a pluriharmonic function on some neighbourhood $U$ of $\SS^3$, i.e., there is a holomorphic function $G:U\to\mathbb{C}$ with $\text{Re}(G)|_{\SS^3}=G_{\mathbb{R}}$, then $\text{Im}(G)|_{\SS^3}$ is a first integral of the Legendrian field $\text{Re}(h(v_1+\I v_2))$.
\end{lemma}
\begin{proof}
Consider the complex structure $J$ on $\mathbb{C}^2\cong\mathbb{R}^4$, which on each real tangent space $T_p\mathbb{R}^4=\mathbb{R}^4$ maps $(a_1,a_2,a_3,a_4)$ to $J(a_1,a_2,a_3,a_4)=(-a_2,a_1,-a_4,a_3)$.
A direct calculation shows that
\begin{align}
\text{Re}(h(z_1,z_2))&=\nabla G_{\mathbb{R}}\cdot v_1=J\nabla G_{\mathbb{R}}\cdot v_2\,,\label{eq.h1}\\
\text{Im}(h(z_1,z_2))&=-\nabla G_{\mathbb{R}}\cdot v_2=J\nabla G_{\mathbb{R}}\cdot v_1\,,\label{eq.h2}
\end{align}
where $\nabla$ denotes the gradient $(\partial_{x_1},\partial_{y_1},\partial_{x_2},\partial_{y_2})$, and $\{v_1,v_2\}$ are the Hopf fields introduced in Section~\ref{sec:bateman}. It follows that
\begin{equation}
    \text{Im}(h(v_1+\I v_2))=(J\nabla G_{\mathbb{R}}\cdot v_1)v_1+(J\nabla G_{\mathbb{R}}\cdot v_2)v_2\,,
\end{equation}
or equivalently, $\text{Im}(h(v_1+\I v_2))$ is the projection of $J\nabla G_{\mathbb{R}}$ into the contact plane at each point of $\SS^3$, spanned by $v_1$ and $v_2$.

Now we change the basis of each tangent space from ${\partial_{x_1}, \partial_{y_1}, \partial_{x_2}, \partial_{y_2}}$ to
\begin{equation}
    \{v_1,v_2,v_3,v_4\}:=\{v_1,v_2,x_1\partial_{x_1}+y_1\partial_{y_1}+x_2\partial_{x_2}+y_2\partial_{y_2},-y_1\partial_{x_1}+x_1\partial_{y_1}-y_2\partial_{x_2}+
    x_2\partial_{y_2}\}\,.
\end{equation}
Note that this defines an orthonormal basis at each point of $\SS^3$, and the change of basis preserves orthogonality. Further note that $J(v_1)=v_2$ and $J(v_3)=v_4$, so that in this new basis $J$ still takes the form $J(a_1,a_2,a_3,a_4)=(-a_2,a_1,-a_4,a_3)$. So if $\nabla G_{\mathbb{R}}=(a_1,a_2,a_3,a_4)$ in this basis, then $J(\nabla G_{\mathbb{R}})=(-a_2,a_1,-a_4,a_3)$ and $\text{Im}(h(v_1+\I v_2))=(-a_2,a_1,0,0)$, which is clearly orthogonal to $\nabla G_{\mathbb{R}}$. Therefore, $G_{\mathbb R}$ is a first integral of the vector field $\text{Im}(h(v_1+\I v_2))$.

If $G_\mathbb{R}$ is the real part of a holomorphic function $G$, Equation~\eqref{eq:htang} becomes
\begin{equation}\label{eq.fig}
    h(z_1,z_2)=2\frac{\partial G}{\partial z_2}\overline{z_1}-2\frac{\partial G}{\partial z_1}\overline{z_2}
\end{equation}
and
\begin{equation}
    J(\nabla G_{\mathbb{R}})=J(\nabla \text{Re}(G))=\nabla\text{Im}(G)\,.
\end{equation}
Therefore, we obtain
\begin{align}
    \text{Re}(h(v_1+\I v_2))\cdot\nabla\text{Im}(G)&=J(\text{Re}(h(v_1+\I v_2)))\cdot J(\nabla\text{Im}(G))\nonumber\\
    &=-\text{Im}(h(v_1+\I v_2))\cdot \nabla(G_{\RR})=0\,,
\end{align}
thus implying that $\text{Im}(G)$ is a first integral of the vector field $\text{Re}(h(v_1+\I v_2))$.
\end{proof}

Next we use this lemma to provide a criterion for the existence of invariant surfaces for Legendrian fields on $\SS^3$. We recall that a level set of a function is regular if the gradient of the function does not vanish at any point of the level set.

\begin{proposition}\label{P.arg}
A Legendrian field $X=\text{Im}(\Th(v_1+\I v_2))$ with $\Th:\SS^3\to\mathbb{C}$, is tangent to the regular level set $M=G_{\mathbb{R}}^{-1}(c)$, $c\in\mathbb{R}$, of a real function $G_{\mathbb{R}}:\SS^3\to\mathbb{R}$ if and only if
\begin{equation}
    \arg(\Th(z_1,z_2))=\arg(h(z_1,z_2))
\end{equation}
for all $(z_1,z_2)\in M$ with $h(z_1,z_2)\neq 0$, where $h$ is given by Equation~\eqref{eq:htang}.
\end{proposition}
\begin{proof}
It follows from Equations~\eqref{eq.h1} and~\eqref{eq.h2} that $h(z_1,z_2)=0$ if and only if $\nabla G_{\mathbb{R}}$ is orthogonal to the contact plane at the point $(z_1,z_2)\in\SS^3$. In particular, any Legendrian vector field (with respect to the standard contact form) is tangent to $M$ at a point $(z_1,z_2)$ if $h(z_1,z_2)=0$. Next we analyze the points $(z_1,z_2)\in M$ such that $h(z_1,z_2)\neq 0$.

We have already seen, cf. Lemma~\ref{L.fi}, that $\text{Im}(h(v_1+\I v_2))$ is tangent to $M$. It then follows that if $\arg(\Th(z_1,z_2))=\arg(h(z_1,z_2))$ for some $(z_1,z_2)\in M$, then $X$ is tangent to $M$ at $(z_1,z_2)$. Conversely, if $h(z_1,z_2)\neq 0$, the intersection of the contact plane at $(z_1,z_2)$ and $T_{(z_1,z_2)}M$ is 1-dimensional (because the level set is regular), so that any Legendrian tangent vector to $M$ has to be a real multiple of $\text{Im}(h(v_1+\I v_2))$ at $(z_1,z_2)$, and hence $\arg(\Th(z_1,z_2))=\arg(h(z_1,z_2))$.
\end{proof}

In particular, the following corollary is straightforward from Bateman construction and Proposition~\ref{P.arg}. It provides a criterion to construct null electromagnetic fields with prescribed invariant surfaces.

\begin{corollary}
Let $G_{\mathbb{R}}:\SS^3\to\mathbb{R}$ be a smooth function. Then there is a Bateman electromagnetic field of Hopf type, whose magnetic part at $t=0$ (when projected onto $\SS^3$) is tangent to the regular level set $M:=G_{\mathbb{R}}^{-1}(c)$, $c\in\RR$, if and only if there is a holomorphic function $\Th:U\to \mathbb{C}$ on some neighbourhood $U$ of $\SS^3$ with
\begin{equation}\label{eq:th}
    \arg(\Th(z_1,z_2))=\arg(h(z_1,z_2))
\end{equation}
for all $(z_1,z_2)\in M$ with $h(z_1,z_2)\neq 0$, where $h$ is given by Equation~\eqref{eq:htang}.
\end{corollary}

\subsection{Totally tangential $\CC$-links and Bateman}\label{SS.tot}

Let us start with the following definition:

\begin{definition}[Rudolph \cite{rudolph}]\label{def:tt}
A link $L$ in $\SS^3$ is called a totally tangential $\mathbb{C}$-link if there exists a holomorphic function $G:\mathbb{C}^2\to\mathbb{C}$ such that
\begin{itemize}
    \item $G^{-1}(0)\cap\SS^3=L$.
    \item $G^{-1}(0)\cap \mathring{D}^4=\emptyset$, where $\mathring{D}^4$ denotes the open 4-ball bounded by $\SS^3$ in $\CC^2$.
    \item $L$ is a non-degenerate critical manifold of index 1 of $\rho|_{\text{Reg}(G)}$, the restriction of $\rho(z_1,z_2):=|z_1|^2+|z_2|^2$ to the set of regular points of $G$.
\end{itemize}
\end{definition}
\begin{remark}
It is known that a link in $\SS^3$ is a totally tangential $\mathbb{C}$-link if and only if it is real-analytic and Legendrian~\cite{rudolphtt, rudolphtt2, rudolph}.
\end{remark}

In~\cite{bode} it was suggested that there could be a relation between the holomorphic function $G$ that defines $L$ as an intersection and the holomorphic function $h$, which defines the electromagnetic field containing $L$ (via Bateman construction). Next we shall show that in the case of torus knots~\cite{kbbpsi} we can make this relation precise, in a way that incorporates the integrability of the electromagnetic fields.

We first observe that the polynomial $G$ that defines the torus knot $T_{p,-q}$ as a totally tangential intersection of $G^{-1}(0)$ and $\SS^3$ is given by $$G(z_1,z_2)=\rho z_1^p z_2^q-1$$
for some appropriately chosen constant $\rho>0$ (see~\cite{rudolphtt}). For the class of electromagnetic solutions found in~\cite{kbbpsi}, the following result establishes a connection between the function $G$ and the first integrals of the fields.

\begin{proposition}\label{P.Gh}
The Legendrian field on $\SS^3$ that corresponds to the electric field $E$ in~\cite{kbbpsi} (containing the torus knot $T_{p,q}$ in $\mathbb{R}^3$, cf. Remark~\ref{rem}) is tangent to the level sets of $\text{Im}(G)$, and the Legendrian field corresponding to the associated magnetic field $B$ is tangent to the level sets of $\text{Re}(G)$.
\end{proposition}

\begin{proof}
Indeed, the Bateman fields constructed in~\cite{kbbpsi} use the complex polynomial
$$\tilde h(z_1,z_2)=pqz_1^{p-1}z_2^{q-1}\,,$$
while the $(p,q)$-torus knot arises as a totally tangential $\mathbb{C}$-link for the function $G$ introduced before. In order to construct a pair of Legendrian fields on $\SS^3$ for which $\text{Re}(G)$ and $\text{Im}(G)$ are first integrals, we use Lemma~\ref{L.fi} and define $h$ as in Equation~\eqref{eq.fig}, resulting in
\begin{equation}
h(z_1,z_2)=2\rho q z_1^p z_2^{q-1}\overline{z_1}-2\rho p z_1^{p-1}z_2^q\overline{z_2}=pq z_1^{p-1}z_2^{q-1}\Big(2\rho\frac{1}{p}|z_1|^2-2\rho\frac{1}{q}|z_2|^2\Big)\,.
\end{equation}
Ignoring the real factor $2\rho\frac{1}{p}|z_1|^2-2\rho\frac{1}{q}|z_2|^2$ we are left with the same holomorphic function $\tilde h$ as in~\cite{kbbpsi}. Accordingly, up to a real factor, the electric and magnetic parts of the Bateman field~\cite{kbbpsi}, when projected onto $\SS^3$, are tangent to the Legendrian fields obtained in Lemma~\ref{L.fi} using the function $h$. We then conclude that the magnetic and electric lines are tangent to the level sets of the real and imaginary part of $G$, respectively, as we wanted to prove.
\end{proof}

\begin{remark}
There is a topological obstruction regarding the link types that can arise as periodic orbits of non-vanishing vector fields with analytic first integrals. This class of links is called zero-entropy links and has been classified by Etnyre and Ghrist, following previous work by Wada~\cite{wada,EG}. Therefore, since $G$ is a holomorphic function, if the construction above generalizes to other link types, it must produce Bateman fields with zeros in the case of links that do not belong to the family of zero-entropy links.
\end{remark}

To conclude this section we want to emphasize that it would be interesting to see if the construction in Proposition~\ref{P.Gh} generalizes to other holomorphic functions $G$, resulting in other Bateman type solutions exhibiting first integrals. However, at this stage it remains difficult to determine the existence of a holomorphic function $\Th$ satisfying Equation~\eqref{eq:th}. In fact, the torus knots above are the only examples where the function $G$ is known explicitly (although it exists for every link type). Likewise, we know that there exist Bateman fields that realize electric and magnetic lines of any link type, but explicit expressions have only been found for torus links~\cite{kbbpsi}.

\section*{Acknowledgements}

The authors are grateful to R. Casals and A. del Pino for useful comments concerning Legendrian fields. B.B. is supported by the European Union's Horizon 2020 research and innovation programme through the Marie Sklodowska-Curie grant agreement 101023017.
This work is supported by the grants CEX2019-000904-S and PID2019-106715GB GB-C21 funded by MCIN/AEI/10.13039/501100011033.

\bibliographystyle{amsplain}

\end{document}